\documentclass[a4paper,11pt]{article}
\usepackage{stmaryrd}
\usepackage{amsfonts}
\usepackage{bbm}
\usepackage{amscd}
\usepackage{mathrsfs}
\usepackage{latexsym,amssymb,amsmath,amscd,amscd,amsthm,amsxtra,xypic}
\usepackage[dvips]{graphicx}
\usepackage[utf8]{inputenc}
\usepackage[T1]{fontenc}
\usepackage{enumerate}
\usepackage{lmodern}
\usepackage{amssymb}
\usepackage[all]{xy}
\usepackage{ifpdf}
\ifpdf
\usepackage[colorlinks=true,linkcolor=blue,final,backref=page,hyperindex,citecolor=red]{hyperref}
\else
\usepackage[colorlinks,final,backref=page,hyperindex,hypertex]{hyperref}
\fi

\usepackage{nicefrac,mathtools,enumitem}
\usepackage{microtype}
\usepackage{amsfonts}
\usepackage{amssymb}
\usepackage{mathrsfs}
\usepackage{amsmath}
\usepackage{amsthm}
\usepackage{enumerate}
\usepackage{indentfirst}
\usepackage{amsfonts}
\usepackage{amssymb}
\usepackage{mathrsfs}
\usepackage{amsmath}
\usepackage{amsthm}
\usepackage{enumerate}
\usepackage{cite}
\usepackage{mathrsfs}
\usepackage{geometry}
\allowdisplaybreaks
\newtheorem{defn}{Definition}[section]
\newtheorem{thm}[defn]{Theorem}
\newtheorem{lem}[defn]{Lemma}
\newtheorem{prop}[defn]{Proposition}
\newtheorem{cor}[defn]{Corollary}
\newtheorem{ex}[defn]{Example}
\newtheorem{re}[defn]{Remark}

\def\c{\cdot}

\numberwithin{equation}{section}

\newcommand{\emptycomment}[1]{}

\begin{document}

\title{\sf $D$-bialgebras, dendrification  and embeddings into 
{\sf AWB} of almost Poisson algebras}
\author{  \bf Sami Mabrouk\footnote {  E-mail: sami.mabrouk@fsgf.u-gafsa.tn, mabrouksami00@yahoo.fr} }
\date{{ Faculty of Sciences, University of Gafsa,    Tunisia}}
 
\maketitle


\abstract{An algebra with bracket ({\sf AWB} for short) is an associative algebra endowed with a bilinear bracket satisfying a Leibniz-type compatibility condition, as introduced in \cite{casas}. It can be viewed as a noncommutative generalization of an almost Poisson algebra; indeed, when the associative product is commutative and the bracket is skew-symmetric, one recovers the notion of an almost Poisson algebra.
In this paper, we introduce the notion of  {almost Poisson Drinfel'd bialgebras ($D$-bialgebras)} as an analogue of Poisson $D$-bialgebras, and we establish the equivalence between matched pairs, Manin triples, and almost Poisson $D$-bialgebras. Furthermore, we define a new algebraic structure, called  {almost tridendriform  Poisson algebras}, which can be regarded as the underlying algebraic structures associated with relative Rota-Baxter operators on almost Poisson algebras.
Finally, we show that every almost Poisson algebra can be embedded into an algebra with bracket ({\sf AWB}) via averaging operators, and more generally via relative averaging operators associated to a given representation of the almost Poisson algebra.} 

\vspace{0.3cm}

\vspace{0.3cm}

\tableofcontents
\section{Introduction}
This paper deals with associative algebras $A$ endowed with a bilinear operation 
$[\cdot,\cdot]: A \otimes A \to A$ which is a left biderivation (respectively, a right biderivation), that is, for all $x,y,z \in A$,
\begin{align}
[xy,z] &= x[y,z] + [x,z]y, \label{Eq1}\\
(\text{resp. } \ [x,yz] &= y[x,z] + [x,y]z). \label{Eq2}
\end{align}

Such an algebraic structure is called an \emph{algebra with bracket} ({\sf AWB}), introduced in \cite{casas}. Its origins can be traced back to the physics literature (see \cite{Kanatchikov}). The prototypical example of this structure is given by Poisson algebras. Further examples, as well as homological properties of {\sf AWB}s, can be found in \cite{Casas1, Casas2, Casas3, Casas4}.

The aim of this paper is to study a particular class of {\sf AWB}s in which the associative product is commutative and the bracket $[\cdot,\cdot]$ is skew-symmetric. In this case, it is sufficient to impose only one of the identities \eqref{Eq1} or \eqref{Eq2}. 
This class of algebras is known under several different names. To the best of our knowledge, it was introduced independently by Cannas da Silva and Weinstein in \cite{CannasWeinstein} under the name \emph{almost Poisson algebras}, and by Shestakov in \cite{Shestakov} under the name \emph{general Poisson algebras} (later renamed \emph{generic Poisson algebras} in \cite{KolesnikovMakarLimanov}). A fundamental example is given by the algebra of smooth functions on a manifold $M$, endowed with the usual product and the bracket defined by
$
[f,g] = \Lambda(df,dg),
$
where $\Lambda$ is a bivector field on $M$. The term ``almost'' emphasizes that such algebras are not necessarily Poisson algebras.

The study of operators on algebraic structures is a classical and well-established topic in algebra. Among the most fundamental examples are endomorphisms and automorphisms, which are typically defined on a generating set of an algebra and then extended to the whole structure. In contrast, other classes of operators exhibit a more intricate construction. 

Rota--Baxter operators were first introduced in the seminal work of Baxter \cite{Bax} (1960) in the context of fluctuation theory. Since then, their theory has been extensively developed across various areas of mathematics. Given an algebra $A$, a linear map $\mathcal{R}: A \to A$ is called a Rota--Baxter operator of weight $\lambda$ if, for all $x,y \in A$,
\begin{equation}
\mathcal{R}(x)\mathcal{R}(y)
= \mathcal{R}\big(x\mathcal{R}(y)\big)
+ \mathcal{R}\big(\mathcal{R}(x)y\big)
+ \lambda \mathcal{R}(xy).
\end{equation}
These operators are important due to their deep connections with several key structures, including the Yang--Baxter equation \cite{MR0674005, MR0725413, MR2568415}, dendriform algebras \cite{MR3021790}, and the renormalization of quantum field theory \cite{CK}. A comprehensive overview of their history and connections can be found in \cite{G}.

Other important classes of operators include Reynolds operators and averaging operators, which have appeared independently in different areas of mathematics. Let $E$ be a locally compact Hausdorff space and denote by $C_0(E)$ the Banach algebra of continuous functions on $E$ vanishing at infinity. A Reynolds operator $\mathcal{K}: C_0(E) \to C_0(E)$ (see \cite{R}) is a linear map satisfying the Reynolds identity
\[
\mathcal{K}(x)\mathcal{K}(y)
= \mathcal{K}\big(\mathcal{K}(x)y + x\mathcal{K}(y) - \mathcal{K}(x)\mathcal{K}(y)\big),
\qquad \forall x,y \in C_0(E).
\]
In particular, any projection $\mathcal{K}$ is a Reynolds operator if it is averaging, that is,
\[
\mathcal{K}(x)\mathcal{K}(y)
= \mathcal{K}\big(\mathcal{K}(x)y\big),
\qquad \forall x,y \in C_0(E).
\]
Reynolds operators were originally introduced in turbulence theory. Averaging operators have been widely studied by Birkhoff \cite{Bir1, Bir2}, Kelley \cite{K}, and Rota \cite{R1, R2}; see also \cite{Safa, HP, FM, KKL, PB, BCEM, S} and the references therein.

In this work, we develop the representation theory of almost Poisson algebras in order to introduce a generalization of weighted relative Rota-Baxter operators ($\mathcal{O}$-operators) on almost Poisson algebras. These operators provide systematic methods for constructing new algebraic structures, such as almost tridendriform  Poisson algebras. The notion of $\mathcal{O}$-operators is a fundamental tool in the splitting of algebraic structures (see \cite{Abdaoui, Chtioui1, Chtioui2, Kupershmidt, Sami, Fatma, Rota} for more details).

Moreover, we introduce and study relative averaging operators (also called embedding tensors) on almost Poisson algebras with respect to a given representation, as a natural generalization of averaging operators. We emphasize their close relationship with {\sf AWB} structures, which is analogous to the connection between Poisson algebras and noncommutative Leibniz--Poisson algebras (see \cite{Casas2}), as discussed in \cite{SafaFatma}.

The following diagram illustrates the relationships between Poisson algebras, almost Poisson algebras, {\sf AWB} structures, and noncommutative Leibniz--Poisson algebras:
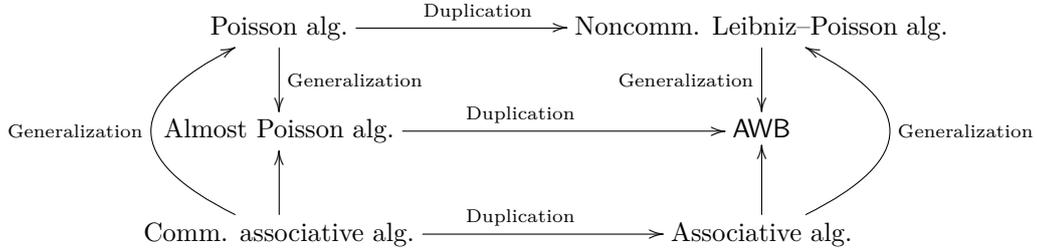
\begin{equation*} 
\xymatrix{
\text{Poisson alg.} \ar[d]^{\text{Generalization}}
\ar[rr]^{\text{Duplication\qquad\qquad}} & 
& \text{Noncomm. Leibniz--Poisson alg.} \ar[d]_{\text{Generalization}}  \\
\text{Almost Poisson alg.} \ar[rr]^{\text{Duplication}}  
&   & \text{\sf AWB}   \\
\text{Comm. associative alg.} \ar[u] \ar@/^4pc/[uu]^{\text{Generalization}} 
\ar[rr]^{\text{Duplication}} 
&    &
\text{Associative alg.} \ar[u] \ar@/_4pc/[uu]_{\text{Generalization}}   
}
\end{equation*}

A bialgebra structure consists of a comultiplication together with compatibility conditions between the multiplication and the comultiplication. For instance, if $V$ is a finite-dimensional vector space endowed with a given algebraic structure, one approach is to equip the dual space $V^*$ with a compatible structure and impose suitable compatibility conditions between $V$ and $V^*$. 

A major motivation for studying Lie bialgebras is the existence of a coboundary theory, which allows one to construct Lie bialgebras from solutions of the classical Yang--Baxter equation. Originally arising in physics, the Yang--Baxter equation plays a central role in this theory. Lie bialgebras were introduced by Drinfeld \cite{drinfeld} and have important applications in symplectic geometry and quantum groups; they are also equivalent to Manin triples. It is well known that Lie bialgebras and pre-Lie bialgebras arise from the study of antisymmetric and symmetric bilinear forms on Lie algebras, respectively.

The notion of infinitesimal $D$-bialgebra was introduced in \cite{Aguiar2} as an analogue of Manin triples for Lie algebras. In this paper, we investigate a Drinfeld-type construction in order to introduce almost Poisson Drinfeld bialgebras (or $D$-bialgebras).

~~

 \textbf{Organization.}
The paper is organized as follows. In Section~\ref{Section2}, we recall the notion of matched pairs of almost Poisson algebras. Section~\ref{Section3} is devoted to the theory of $D$-bialgebras for almost Poisson algebras, developed via the frameworks of matched pairs and Manin triples. These structures are formulated in terms of commutative infinitesimal $D$-bialgebras together with suitable compatibility conditions.
In Section~\ref{Section4}, we introduce weighted relative Rota-Baxter operators on almost Poisson algebras and on almost tridendriform Poisson algebras, thereby extending the notion of post-Poisson algebras. More precisely, a weighted relative Rota-Baxter operator on an almost Poisson algebra gives rise to an almost tridendriform Poisson algebra, while conversely, an almost tridendriform Poisson algebra naturally induces a weighted relative Rota-Baxter operator on the associated almost Poisson algebra.
Finally, Section~\ref{Section5} is devoted to the study of relative averaging operators (also called embedding tensors) on almost Poisson algebras with respect to a suitable representation, leading to {\sf AWB} structures. We further provide characterizations of these operators in terms of their graphs and Nijenhuis operators.

~~

 \textbf{Notations and convention.}
\begin{enumerate}
    \item Throughout this paper, $\mathbb{K}$ is a field of characteristic zero and all vector spaces, tensor product, and algebras are finite dimensional over $\mathbb{K}$.
\item Let $V$ and $W$ be two vector spaces:
\begin{enumerate}
    \item Denote by $\tau:V\otimes W \rightarrow W\otimes V$ the switch  isomorphism, $\tau(v\otimes w)=w\otimes v$.
    \item For a linear map $\Delta: V \rightarrow \otimes^{2}V$ , we use Sweedler's notation $\Delta(x) = \sum_{(x)}
 x_1\otimes 
 x_2$ for $x \in V$. We will often omit the summation sign $\sum_{(x)}$
to simplify the notations.
\end{enumerate}
\end{enumerate}
\section{Matched pairs of  almost Poisson algebras}\label{Section2}
In this section, we study the basic definitions of commutative associative algebras and almost Poisson algebras, their representations, dual representations and matched pairs (see \cite{Bai2010, Shestakov, Kuzmin} for more details). 
 
\subsection{Almost Poisson Algebras and Representations}
 
\begin{defn}
An \textbf{associative algebra} is a pair $(A,\cdot)$, where $A$ is a vector space and $\cdot : A\times A \to A$ is a bilinear map satisfying the \textbf{associativity condition}:
$$
(x\cdot y)\cdot z = x \cdot (y\cdot z), \quad \forall\, x,y,z \in A.
$$
The algebra is called \textbf{commutative} if, in addition,
$$
x\cdot y = y\cdot x, \quad \forall\, x,y \in A.
$$
\end{defn}
 
\begin{defn}
Let $(A,\cdot)$ be an associative algebra. A \textbf{representation} (or \textbf{module}) of $A$ is a pair $(V, \mu)$, where $V$ is a vector space and $\mu:A\to \mathrm{End}(V)$ is a linear map satisfying
$$
\mu(x\cdot y)=\mu(x)\mu(y), \quad \forall\, x,y \in A.
$$
The map $\mu$ is also called a \textbf{representation of $A$ on the vector space $V$}.
\end{defn}
 
\begin{ex}
If $(A,\cdot)$ is an associative algebra and $L: A \to \mathrm{End}(A)$ is defined by
$$
L(x)(y) = x \cdot y, \quad \forall\, x, y \in A,
$$
then $L(x):A\to A$ is a linear map on $A$ satisfying
$$
L(x\cdot y)(z)=(x\cdot y)\cdot z=x\cdot (y\cdot z)=L(x)L(y)(z), \quad \forall\, x,y,z\in A,
$$
which is equivalent to $L(x\cdot y)=L(x)L(y)$ as an operator equality in $\mathrm{End}(A)$. Thus, $(A,L)$ is a representation of $(A,\cdot)$, called the \textbf{regular representation} of $(A,\cdot)$.
\end{ex}

Let $V$ be a vector space. For any linear map $\theta : A \to \mathrm{End}(V)$, we define the \textbf{dual linear map} $\theta^*:A \to \mathrm{End}(V^*)$ by
\begin{equation} \label{defeq-dualrepMA}
\langle \theta^*(x)\xi,v\rangle=-\langle\xi,\theta(x)v\rangle,\quad \forall\, x\in A,\,\xi\in V^*,\,v\in V,
\end{equation}
where $\langle \cdot,\cdot \rangle$ denotes the canonical pairing between $V$ and $V^*$.
 
\begin{lem}\label{lem:dualrep}
If $(V,\mu)$ is a representation of a commutative associative algebra $(A, \cdot)$, then $(V^*,-\mu^*)$ is also a representation of $(A, \cdot)$. In particular, $(A^*,-L^*)$ is a representation of $(A, \cdot)$.
\end{lem}

\begin{defn}
An \textbf{almost Poisson algebra} (also called  $(A,[\cdot,\cdot], \cdot)$ is a vector space $A$ equipped with  bilinear operation $\cdot:A\times A \to A$ and a skewsymmetric bracket $[\cdot,\cdot]:A\times A \to A$ such that:
\begin{enumerate}
\item $(A,\cdot)$ is a commutative associative algebra,
\item The \textbf{Leibniz rule} is satisfied:
\begin{equation} \label{eq:malpoisalg}
[x,y\cdot z]=[x, y] \cdot z+ y\cdot [x,z], \quad \forall\, x,y,z\in A.
\end{equation}
\end{enumerate}
\end{defn}
 
\begin{defn}
Let $(A, [\cdot,\cdot], \cdot)$ be an almost Poisson algebra, let $V$ be a vector space, and let $\varrho, \mu : A \to \mathrm{End}(V)$ be linear maps. Then $(V, \varrho, \mu)$ is called a \textbf{representation} (or \textbf{module}) of $A$ if $(V, \mu)$ is a representation of the associative algebra $(A, \cdot)$ and the following compatibility conditions hold for all $x, y \in A$:
\begin{align}
\varrho(x\cdot y)&=\mu(y)\varrho(x)+\mu(x)\varrho(y),\label{eq:repmalcpoisalg1}\\
\mu([x,y])&=\varrho(x)\mu(y)-\mu(y)\varrho(x).\label{eq:repmalcpoisalg2}
\end{align}
\end{defn}
 
\begin{thm}\label{thm:semidirect}
Let $(A, [\cdot,\cdot], \cdot)$ be an almost Poisson algebra and $(V, \varrho, \mu)$ a representation of $A$. Then the direct sum $A \oplus V$ becomes an almost Poisson algebra under the operations (still denoted by $[\cdot,\cdot]$ and $\cdot$):
\begin{align}
[x_1+v_1,x_2+v_2]&=[x_1,x_2]+\varrho(x_1)v_2-\varrho(x_2)v_1,\label{eq:semidirectbracket}\\
(x_1+v_1)\cdot(x_2+v_2)&=x_1\cdot x_2+ \mu(x_1)v_2+ \mu(x_2)v_1,\label{eq:semidirectproduct}
\end{align}
for all $x_1, x_2 \in A$ and $v_1,v_2 \in V$. We denote this algebra by $A\ltimes V$.
\end{thm}
 
\begin{proof}
Straightforward verification using the axioms of representations.
\end{proof}
 
\begin{lem}\label{lem:dualrepAP}
If $(V, \varrho, \mu)$ is a representation of an almost Poisson algebra $(A, [\cdot,\cdot], \cdot)$, then $(V^*, \varrho^*, -\mu^*)$ is also a representation of $(A, [\cdot,\cdot], \cdot)$, where $\varrho^*, \mu^*$ are defined as in \eqref{defeq-dualrepMA}. Therefore, both $(A,\mathrm{ad}, L)$ and $(A^*,\mathrm{ad}_A^*, -L^*)$ are representations of the almost Poisson algebra $(A, [\cdot,\cdot], \cdot)$.
\end{lem}
 
\subsection{Matched Pairs of Almost Poisson Algebras}
 
We now recall the notion of matched pairs of commutative associative algebras introduced by C. Bai \cite{Bai2010}.
 
\begin{defn}\label{defn:matchedpairCAA}
Let $(A_1,\cdot_1)$ and $(A_2,\cdot_2)$ be commutative associative algebras, and let $\mu_1:A_1\to \mathrm{End}(A_2)$ and $\mu_2:A_2\to \mathrm{End}(A_1)$ be linear maps such that $(A_2,\mu_1)$ and $(A_1,\mu_2)$ are representations of $A_1$ and $A_2$, respectively. If for all $x_1,y_1 \in A_1$ and $x_2,y_2 \in A_2$, the following compatibility conditions hold:
\begin{align}
\mu_1(x_1)(x_2\cdot_2 y_2)&=\bigl(\mu_1(x_1)x_2\bigr)\cdot_2 y_2+\mu_1\bigl(\mu_2(x_2)x_1\bigr)y_2,\label{eq:matchCAA1}\\
\mu_2(x_2)(x_1\cdot_1 y_1)&=\bigl(\mu_2(x_2)x_1\bigr)\cdot_1 y_1+\mu_2\bigl(\mu_1(x_1)x_2\bigr)y_1,\label{eq:matchCAA2}
\end{align}
then $(A_1,A_2, \mu_1,\mu_2)$ is called a \textbf{matched pair of commutative associative algebras}.
\end{defn}
 
\begin{prop}[\cite{Bai2010}]\label{prop:matchedpairCAA}
Let $(A_1,\cdot_1)$ and $(A_2,\cdot_2)$ be commutative associative algebras, and let $\mu_1:A_1\to \mathrm{End}(A_2)$ and $\mu_2:A_2\to \mathrm{End}(A_1)$ be linear maps. Define a bilinear map $"\cdot"$ on $A_1\oplus A_2$ by
\begin{equation}\label{Matass}
(x_1+x_2)\cdot(y_1+y_2)=x_1 \cdot_1 y_1+\mu_2(x_2)y_1+\mu_2(y_2)x_1+x_2\cdot_2 y_2+\mu_1(x_1)y_2+\mu_1(y_1)x_2.
\end{equation}
Then $(A_1\oplus A_2, \cdot)$ is a commutative associative algebra if and only if $(A_1,A_2, \mu_1,\mu_2)$ is a matched pair of commutative associative algebras. In this case, we denote this algebra by $A_1\, \bowtie A_2$. Moreover, every commutative associative algebra that is the direct sum of the underlying linear spaces of two subalgebras arises from a matched pair of commutative associative algebras.
\end{prop}
In the following, we study the notion of matched pairs for almost Poisson algebras.
 
Let $\left({A_1, [\cdot,\cdot]_1,\cdot_1}\right)$ and $\left({A_2, [\cdot,\cdot]_2,\cdot_2}\right)$ be almost Poisson algebras, and let
$$
\varrho_1,\mu_1: A_1\to \mathrm{End}(A_2) \quad \text{and} \quad \varrho_2,\mu_2: A_2\to \mathrm{End}(A_1)
$$
be linear maps. On the direct sum $A_1\oplus A_2$, define a bracket $[\cdot,\cdot]: (A_1\oplus A_2) \times (A_1\oplus A_2) \rightarrow A_1\oplus A_2$ by
\begin{equation}\label{MatMal}
[x_1+x_2,y_1+y_2]=[x_1,y_1]_1+\varrho_2(x_2)y_1-\varrho_2(y_2)x_1+[x_2,y_2]_2+\varrho_1(x_1)y_2-\varrho_1(y_1)x_2.
\end{equation}
 
\begin{defn}\label{defn:matchedpairAP}
Let $(A_1, [\cdot,\cdot]_1, \cdot_1)$ and $(A_2, [\cdot,\cdot]_2, \cdot_2)$ be two almost Poisson algebras. Let $\varrho_1, \mu_1: A_1\to \mathrm{End}(A_2)$ and $\varrho_2, \mu_2: A_2\to \mathrm{End}(A_1)$ be four linear maps such that:
\begin{enumerate}
\item $(A_1, A_2, \mu_1,\mu_2)$ is a matched pair of commutative associative algebras,
\item $(A_2, \varrho_1,\mu_1)$ and $(A_1,\varrho_2,\mu_2)$ are representations of the almost Poisson algebras $(A_1, [\cdot,\cdot]_1, \cdot_1)$ and $(A_2, [\cdot,\cdot]_2, \cdot_2)$, respectively,
\item The maps $\varrho_1,\varrho_2,\mu_1,\mu_2$ satisfy the following \textbf{compatibility conditions} for all $x_1, y_1 \in A_1$ and $x_2,y_2 \in A_2$:
\begin{align}
\varrho_2(x_2)(x_1\cdot_1 y_1)&=\bigl(\varrho_2(x_2)x_1\bigr)\cdot_1 y_1+x_1\cdot_1\bigl(\varrho_2(x_2)y_1\bigr)\nonumber\\
&\quad -\mu_2\bigl(\varrho_1(x_1)x_2\bigr)y_1-\mu_2\bigl(\varrho_1(y_1)x_2\bigr)x_1,\label{matpoi1}\\[0.2cm]
\varrho_1(x_1)(x_2\cdot_2 y_2)&=\bigl(\varrho_1(x_1)x_2\bigr)\cdot_2 y_2+x_2\cdot_2\bigl(\varrho_1(x_1)y_2\bigr)\nonumber\\
&\quad -\mu_1\bigl(\varrho_2(x_2)x_1\bigr)y_2-\mu_1\bigl(\varrho_2(y_2)x_1\bigr)x_2,\label{matpoi2}\\[0.2cm]
\bigl[x_1,\mu_2(x_2)y_1\bigr]_1-\varrho_2\bigl(\mu_1(y_1)x_2\bigr)x_1&=\mu_2\bigl(\varrho_1(x_1)x_2\bigr)y_1-\bigl(\varrho_2(x_2)x_1\bigr)\cdot_1 y_1\nonumber\\
&\quad +\mu_2(x_2)[x_1,y_1]_1,\label{matpoi3}\\[0.2cm]
\bigl[x_2,\mu_1(x_1)y_2\bigr]_2-\varrho_1\bigl(\mu_2(y_2)x_1\bigr)x_2&=\mu_1\bigl(\varrho_2(x_2)x_1\bigr)y_2-\bigl(\varrho_1(x_1)x_2\bigr)\cdot_2 y_2\nonumber\\
&\quad +\mu_1(x_1)[x_2,y_2]_2.\label{matpoi4}
\end{align}
\end{enumerate}
Such a structure is called a \textbf{matched pair of almost Poisson algebras} $A_1$ and $A_2$. We denote it by $(A_1, A_2, \varrho_1, \mu_1,\varrho_2,\mu_2)$.
\end{defn}
 
\begin{prop}\label{prop:matchedpair}
Let $(A_1, [\cdot,\cdot]_1, \cdot_1)$ and $(A_2, [\cdot,\cdot]_2, \cdot_2)$ be two almost Poisson algebras. For linear maps $\varrho_1, \mu_1: A_1\to \mathrm{End}(A_2)$ and $\varrho_2, \mu_2: A_2\to \mathrm{End}(A_1)$, define operations $"\cdot"$ and $[\cdot,\cdot]$ on $A_1\oplus A_2$ by \eqref{Matass} and \eqref{MatMal}, respectively. Then $(A_1\oplus A_2, [\cdot,\cdot], \cdot)$ is an almost Poisson algebra if and only if $(A_1, A_2, \varrho_1, \mu_1,\varrho_2,\mu_2)$ is a matched pair of almost Poisson algebras. In this case, we denote this almost Poisson algebra by $A_1\, \bowtie A_2$. Moreover, every almost Poisson algebra that is the direct sum of the underlying spaces of two subalgebras can be obtained from a matched pair of almost Poisson algebras.
\end{prop}
 
\begin{proof}
It is known   that $(A_1\oplus A_2, \cdot)$ is a commutative associative algebra if and only if $(A_1, A_2, \mu_1,\mu_2)$ is a matched pair of commutative associative algebras. 
Let $x_1,y_1,z_1 \in A_1$ and $x_2, y_2,z_2 \in A_2$. We verify the Leibniz rule \eqref{eq:malpoisalg} on $A_1 \oplus A_2$. By expanding using \eqref{Matass} and \eqref{MatMal}, we compute:
\begin{align*}
&[x_1+x_2,(y_1+y_2)\cdot(z_1+z_2)]-(z_1+z_2)\cdot[x_1+x_2, y_1+y_2]-(y_1+y_2)\cdot[x_1+x_2,z_1+z_2]\\
&=[x_1+x_2,y_1 \cdot_1 z_1+\mu_2(y_2)z_1+\mu_2(z_2)y_1+y_2\cdot_2 z_2+\mu_1(y_1)z_2+\mu_1(z_1)y_2]\\
&\quad -(z_1+z_2)\cdot\bigl([x_1,y_1]_1+\varrho_2(x_2)y_1-\varrho_2(y_2)x_1+[x_2,y_2]_2+\varrho_1(x_1)y_2-\varrho_1(y_1)x_2\bigr)\\
&\quad -(y_1+y_2)\cdot\bigl([x_1,z_1]_1+\varrho_2(x_2)z_1-\varrho_2(z_2)x_1+[x_2,z_2]_2+\varrho_1(x_1)z_2-\varrho_1(z_1)x_2\bigr).
\end{align*}
 
Expanding all terms systematically (using the fact that $(A_1,[\cdot,\cdot]_1,\cdot_1)$ and $(A_2,[\cdot,\cdot]_2, \cdot_2)$ satisfy the Leibniz rule, and that $(A_2, \varrho_1,\mu_1)$ and $(A_1,\varrho_2,\mu_2)$ are representations), we find that the Leibniz rule on $A_1 \oplus A_2$ is equivalent to the compatibility conditions \eqref{matpoi1}--\eqref{matpoi4}. Hence the conclusion holds.
\end{proof}
\section{Almost Poisson $D$-bialgebras}
 \label{Section3}
In this section, we introduce the concepts of Manin triples for almost Poisson algebras, standard Manin triples, almost Poisson coalgebras, and almost Poisson $D$-bialgebras. These structures generalize classical notions from Lie and associative algebras to the almost Poisson setting, combining Malcev and commutative associative components with appropriate compatibility conditions.
 
\subsection{Manin Triples of Almost Poisson Algebras}
 
\begin{defn}
A \textbf{Manin triple of almost Poisson algebras} consists of three almost Poisson algebras $(A, [\cdot,\cdot], \cdot)$, $(A^+, [\cdot,\cdot]^+, \cdot^+)$, and $(A^-, [\cdot,\cdot]^-, \cdot^-)$, together with a nondegenerate symmetric bilinear form $\omega(\cdot,\cdot)$ on the underlying vector space of $A$ that is \textbf{invariant} with respect to both operations, in the sense that
\begin{align}
\omega(x \cdot y, z) &= \omega(x, y \cdot z), \label{eq:bilinearform1} \\
\omega([x,y], z) &= \omega(x, [y,z]), \label{eq:bilinearform2}
\end{align}
for all $x,y,z \in A$, and satisfying the following conditions:
\begin{enumerate}[label=(\arabic*)]
\item $A^+$ and $A^-$ are almost Poisson subalgebras of $A$, i.e., they are closed under both $[\cdot,\cdot]$ and $\cdot$,
\item $A = A^+ \oplus A^-$ as vector spaces,
\item $A^+$ and $A^-$ are isotropic with respect to $\omega$, i.e., $\omega(a^+, b^+) = 0 = \omega(a^-, b^-)$ for all $a^+, b^+ \in A^+$ and $a^-, b^- \in A^-$.
\end{enumerate}
\end{defn}
 
A \textbf{homomorphism} between two Manin triples of almost Poisson algebras $(A_1, A_1^+, A_1^-)$ and $(A_2, A_2^+, A_2^-)$, associated with invariant bilinear forms $\omega_1$ and $\omega_2$ respectively, is an almost Poisson algebra homomorphism $\phi: A_1 \to A_2$ such that
\begin{equation*}
\phi(A_1^+) \subseteq A_2^+, \quad \phi(A_1^-) \subseteq A_2^-, \quad \omega_1(x,y) = \omega_2(\phi(x), \phi(y)) \quad \forall x,y \in A_1.
\end{equation*}
If $\phi$ is moreover an isomorphism of vector spaces, the two Manin triples are called \textbf{isomorphic}.
 
A Manin triple of almost Poisson algebras is thus a structure that simultaneously forms a Manin triple for the Malcev algebras $(A, [\cdot,\cdot])$, $(A^+, [\cdot,\cdot]^+)$, $(A^-, [\cdot,\cdot]^-)$ and a ``Manin triple'' for the commutative associative algebras $(A, \cdot)$, $(A^+, \cdot^+)$, $(A^-, \cdot^-)$, sharing the same invariant bilinear form $\omega$ and isotropic subalgebras $A^+$ and $A^-$. Consequently, $A^+$ and $A^-$ are Lagrangian subalgebras of $(A, \omega)$.

A particularly important case is the \textbf{standard Manin triple}, which involves the dual space.
 
\begin{defn}
Let $(A, [\cdot,\cdot], \cdot)$ be an almost Poisson algebra. Suppose there exists an almost Poisson algebra structure on the direct sum $A \oplus A^*$ of the underlying vector space of $A$ and its dual $A^*$ such that $(A \oplus A^*, A, A^*)$ forms a Manin triple of almost Poisson algebras, with respect to the invariant symmetric bilinear form
\begin{equation} \label{eq:stanMalPoi}
\omega_d(x + \xi, y + \eta) = \langle x, \eta \rangle + \langle \xi, y \rangle, \quad \forall x,y \in A, \ \xi,\eta \in A^*,
\end{equation}
where $\langle \cdot, \cdot \rangle$ denotes the canonical pairing between $A$ and $A^*$. Then $(A \oplus A^*, A, A^*)$ is called a \textbf{standard Manin triple of almost Poisson algebras}.
\end{defn}
 
Clearly, a standard Manin triple is a Manin triple. Conversely, every Manin triple of almost Poisson algebras is isomorphic to a standard one. The following structure theorem characterizes standard Manin triples in terms of matched pairs.
 
\begin{thm} \label{thm:matpair}
Let $(A, [\cdot,\cdot]_1, \cdot_1)$ and $(A^*, [\cdot,\cdot]_2, \cdot_2)$ be almost Poisson algebras. Then $(A \oplus A^*, A, A^*)$ is a standard Manin triple of almost Poisson algebras (with respect to $\omega_d$ given by \eqref{eq:stanMalPoi}) if and only if $(A, A^*, \varrho_1, \mu_1, \varrho_2, \mu_2)$ is a matched pair of almost Poisson algebras, where
\begin{equation*}
\varrho_1 = \mathrm{ad}_A^*, \quad \mu_1 = -L^*, \quad \varrho_2 = \mathrm{ad}_{A^*}^*, \quad \mu_2 = -L_{A^*}^*.
\end{equation*}
Here, $\mathrm{ad}_A^*: A^* \to \mathrm{End}(A)$ is the coadjoint representation defined by $\langle \mathrm{ad}_A^*(\xi) x, \eta \rangle = -\langle \xi, [x, \eta] \rangle$ for $\xi, \eta \in A^*$, $x \in A$; $L^*: A^* \to \mathrm{End}(A)$ is the dual of the left multiplication $L$ by $\langle L^*(\xi) x, \eta \rangle = -\langle \xi, x \cdot \eta \rangle$; and similarly for the representations on $A^*$.
\end{thm}
 
\begin{proof}
There exists a commutative associative algebra structure on $A \oplus A^*$ making both $(A, \cdot_1)$ and $(A^*, \cdot_2)$ subalgebras, with $\omega_d$ satisfying \eqref{eq:bilinearform1}, if and only if $(A, A^*, -L^*, -L_{A^*}^*)$ is a matched pair of commutative associative algebras. 
 
Thus, if $(A \oplus A^*, A, A^*)$ is a standard Manin triple of almost Poisson algebras, then by Proposition \ref{prop:matchedpair} (with $A_1 = A$, $A_2 = A^*$, $\mu_1 = -L^*$, $\varrho_1 = \mathrm{ad}_A^*$, $\mu_2 = -L_{A^*}^*$, $\varrho_2 = \mathrm{ad}_{A^*}^*$), the structure $(A, A^*, \mathrm{ad}_A^*, -L^*, \mathrm{ad}_{A^*}^*, -L_{A^*}^*)$ forms a matched pair of almost Poisson algebras. Conversely, if this matched pair holds, then Proposition \ref{prop:matchedpair} yields an almost Poisson algebra structure on $A \oplus A^*$ with $A$ and $A^*$ as subalgebras, and $\omega_d$ is invariant on $A \oplus A^*$. Hence, $(A \oplus A^*, A, A^*)$ is a standard Manin triple.
\end{proof}
 
\subsection{Almost Poisson $D$-bialgebras}

To define almost Poisson $D$-bialgebras, we first introduce the dual notions for coalgebras.
 
\begin{defn}
A \textbf{cocommutative coassociative coalgebra} is a pair $(C, \Delta)$, where $C$ is a vector space and $\Delta: C \to C \otimes C$ is a linear map (the comultiplication) satisfying
\begin{align}
(\tau \circ \Delta) &= \Delta, \label{eq:cocommutative} \\
(\mathrm{id} \otimes \Delta) \circ \Delta &= (\Delta \otimes \mathrm{id}) \circ \Delta, \label{eq:coassociative}
\end{align}
where $\tau: C \otimes C \to C \otimes C$ is the flip map defined by $\tau(c_{(1)}\otimes c_{(2)}) = c_{(2)} \otimes c_{(1)}$ for all $c_{(1)}, c_{(1)} \in C$.

\end{defn}

\begin{defn}
A \textbf{commutative   infinitesimal $D$-bialgebra} is a triple $(A, \cdot, \Delta)$ such that:
\begin{enumerate}[label=(\arabic*)]
\item $(A, \cdot)$ is a commutative associative algebra,
\item $(A, \Delta)$ is a cocommutative coassociative coalgebra,
\item The comultiplication $\Delta$ is infinitesimal with respect to the product $\cdot$, i.e.,
\begin{equation} \label{eq:infinitesimal}
\Delta(x \cdot y) = (L(x) \otimes \mathrm{id}) \Delta(y) + (\mathrm{id} \otimes L(y)) \Delta(x), \quad \forall x,y \in A,
\end{equation}
where $L$ denotes left multiplication.
\end{enumerate}
\end{defn}
 According Sweedler notation Equations \eqref{eq:coassociative} and \eqref{eq:infinitesimal} is equivalent to the following 
\begin{align*}
&x_{(1)}\otimes x_{(2)(1)}\otimes x_{(2)(2)}=x_{(1)(1)}\otimes x_{(1)(2)}\otimes x_{(2)},\\
&(x\c y)_{(1)}\otimes (x\c y)_{(2)}=(x\cdot y_{(1)})\otimes y_{(2)}+ x_{(1)}\otimes (y\cdot x_{(2)}).  
\end{align*}
 
\begin{defn}
Let $C$ be a vector space and $\Delta, \delta: C \to C \otimes C$ linear maps. Then $(C, \Delta, \delta)$ is called an \textbf{almost Poisson coalgebra} if:
\begin{enumerate}[label=(\arabic*)]
\item $(C, \Delta)$ is a cocommutative coassociative coalgebra,
\item The co-Leibniz rule holds:
\begin{equation} 
(\mathrm{id} \otimes \Delta) \delta - (\delta \otimes \mathrm{id}) \Delta - (\tau \otimes \mathrm{id}) (\mathrm{id} \otimes \delta) \Delta = 0. \label{eq:Co1}
\end{equation}
\end{enumerate}
\end{defn}
 The identity \eqref{eq:Co1} is equivalente to via  Sweedler notation 
 $$x_{(1)}\otimes x_{(2)}^{(1)}\otimes x_{(2)}^{(2)}- x_{(1)}^{(1)}\otimes x_{(2)}^{(1)}\otimes x^{(2)}-  x^{(2)}_{(1)}\otimes x^{(1)}\otimes x^{(2)}_{(2)},$$
 where $x\in A$, $\delta(x)=x^{(1)}\otimes x^{(2)}$ and $\Delta(x)=x_{(1)}\otimes x_{(2)}$.
\begin{prop} \label{prop:dual-coalgebra}
The triple $(A, \Delta, \delta)$ is an almost Poisson coalgebra if and only if $(A^*, \Delta^*, \delta^*)$ is an almost Poisson algebra, where $\Delta^*, \delta^*: A^* \otimes A^* \to A^*$ are the adjoints (transposes) of $\Delta, \delta$.
\end{prop}
 
\begin{proof}
The conditions on $\Delta$ and $\delta$ dualize directly: $(A, \Delta)$ is a cocommutative coassociative coalgebra if and only if $(A^*, \Delta^*)$ is a commutative associative algebra, and $(A, \delta)$ is a Malcev coalgebra if and only if $(A^*, \delta^*)$ is a Malcev algebra.
 
For the co-Leibniz rule \eqref{eq:Co1}, let $\xi, \eta, \gamma \in A^*$. Define the product and bracket on $A^*$ by $\xi \cdot \eta = \Delta^*(\xi \otimes \eta)$ and $[\xi, \eta] = \delta^*(\xi \otimes \eta)$. Then, for all $x \in A$,
\begin{align*}
\langle (\mathrm{id} \otimes \Delta) \delta(x), \xi \otimes \eta \otimes \gamma \rangle &= \langle x, \delta^*(\xi \otimes (\eta \cdot \gamma)) \rangle = \langle x, [\xi, \eta \cdot \gamma] \rangle, \\
\langle (\delta \otimes \mathrm{id}) \Delta(x), \xi \otimes \eta \otimes \gamma \rangle &= \langle x, \Delta^*([\xi, \eta] \otimes \gamma) \rangle = \langle x, [\xi, \eta] \cdot \gamma \rangle, \\
\langle (\tau \otimes \mathrm{id}) (\mathrm{id} \otimes \delta) \Delta(x), \xi \otimes \eta \otimes \gamma \rangle &= \langle x, \Delta^*(\eta \cdot [\xi, \gamma]) \rangle = \langle x, \eta \cdot [\xi, \gamma] \rangle.
\end{align*}
Thus, \eqref{eq:Co1} holds if and only if the Leibniz rule \eqref{eq:malpoisalg} holds for $(A^*, [\cdot,\cdot], \cdot) = (A^*, \delta^*, \Delta^*)$.
\end{proof}

\begin{defn} \label{def:almostPoissonDbialgebra}
An \textbf{almost Poisson $D$-bialgebra} is a quintuple $(A, [\cdot,\cdot], \cdot, \Delta, \delta)$ satisfying:
\begin{enumerate}[label=(\arabic*)]
\item $(A, [\cdot,\cdot], \cdot)$ is an almost Poisson algebra,
\item $(A, \Delta, \delta)$ is an almost Poisson coalgebra,
\item $(A, \cdot, \Delta)$ is a commutative and cocommutative infinitesimal $D$-bialgebra (i.e., \eqref{eq:infinitesimal} holds),
\item The following compatibility conditions hold for all $x,y \in A$:
{\small\begin{align}
\delta(x \cdot y) + (\mathrm{ad}(y) \otimes \mathrm{id}) \Delta(x) - (\mathrm{id} \otimes L(x)) \delta(y) + (\mathrm{ad}(x) \otimes \mathrm{id}) \Delta(y) - (\mathrm{id} \otimes L(y)) \delta(x) &= 0, \label{eq:Bi4} \\
\Delta([x,y]) - (L(y) \otimes \mathrm{id}) \delta(x) - (\mathrm{id} \otimes \mathrm{ad}(x)) \Delta(y) + (\mathrm{id} \otimes L(y)) \delta(x) - (\mathrm{ad}(x) \otimes \mathrm{id}) \Delta(y) &= 0. \label{eq:Bi5}
\end{align}}
\end{enumerate}
\end{defn}

These conditions ensure that the bracket and comultiplication, as well as the product and cobracket, interact compatibly, dualizing the representation conditions in the algebra setting. They can be written by Sweedler notations, as follow
 \begin{align*}
(x \cdot y)^{(1)}\times (x \cdot y)^{(2)} + [y,x_{(1)}] \otimes x_{(2)} - y^{(1)} \otimes x\c y^{(1)} + [x,y_{(1)}] \otimes y_{(2)} - x^{(1)} \otimes y\c x^{(1)} &= 0,  \\
([x,y])_{(1)}\otimes ([x,y])_{(2)} - y\cdot x^{(1)} \otimes x^{(2)} - y_{(1)}\otimes [x,y_{(2)}] + x^{(1)} \otimes y\cdot  x^{(2)} - [x,y_{(1)}]\otimes y_{(2)} &= 0 ,
\end{align*}
where $x,y\in A$, $\delta(x)=x^{(1)}\otimes x^{(2)}$ and $\Delta(x)=x_{(1)}\otimes x_{(2)}$.
\begin{ex}

Any Poisson bialgebra (where the bracket is a Lie bracket) is an almost Poisson $D$-bialgebra.

\end{ex}
 
The following theorem establishes the equivalence between almost Poisson $D$-bialgebras and standard Manin triples via matched pairs.
 
\begin{thm} \label{thm:equiv-Dbialgebra}
Let $(A, [\cdot,\cdot], \cdot)$ be an almost Poisson algebra, and suppose $A^*$ carries an almost Poisson algebra structure dual to linear maps $\Delta, \delta: A \to A \otimes A$ (i.e., $\cdot^* = \Delta^*$, $[\cdot,\cdot]^* = \delta^*$). Then the following are equivalent:
\begin{enumerate}[label=(\arabic*)]
\item $(A, [\cdot,\cdot], \cdot, \Delta, \delta)$ is an almost Poisson $D$-bialgebra,
\item $(A, A^*, \mathrm{ad}_A^*, -L^*, \mathrm{ad}_{A^*}^*, -L_{A^*}^*)$ is a matched pair of almost Poisson algebras,
\item $(A \oplus A^*, A, A^*)$ is a standard Manin triple of almost Poisson algebras with respect to $\omega_d$ given by \eqref{eq:stanMalPoi}.
\end{enumerate}
\end{thm}
 
\begin{proof}
By Theorem \ref{thm:matpair}, (2) $\Leftrightarrow$ (3). It remains to show (1) $\Leftrightarrow$ (2). The infinitesimal $D$-bialgebra condition \eqref{eq:infinitesimal} holds if and only if $(A, A^*, -L^*, -L_{A^*}^*)$ is a matched pair of commutative associative algebras. The almost Poisson coalgebra structure on $(A, \Delta, \delta)$ dualizes to the almost Poisson algebra on $A^*$, and the representations $(A^*, -L^*, \mathrm{ad}_A^*)$ and $(A, -L_{A^*}^*, \mathrm{ad}_{A^*}^*)$ follow from Lemma \ref{lem:dualrepAP}.
 
The additional compatibilities \eqref{eq:Bi4}--\eqref{eq:Bi5} are dual to the matched pair conditions \eqref{matpoi1}--\eqref{matpoi4}. Specifically, \eqref{matpoi1} dualizes to \eqref{eq:Bi4} (and \eqref{matpoi2} to the symmetric form), while \eqref{matpoi3}--\eqref{matpoi4} dualize to \eqref{eq:Bi5} and its counterpart. For instance, applying the pairing to \eqref{matpoi1} for $x,y \in A$ and $\xi, \eta \in A^*$ yields
\begin{align*}
&\langle (L(y) \otimes \mathrm{id}) \mathrm{ad}_{A^*}^*(\xi) x + (L(x) \otimes \mathrm{id}) \mathrm{ad}_{A^*}^*(\xi) y + (-L_{A^*}^*) (\mathrm{ad}_{A^*}^*(x) \xi) y + (-L_{A^*}^*) (\mathrm{ad}_{A^*}^*(y) \xi) x - \mathrm{ad}_{A^*}^*(\xi) (x \cdot y), \eta \rangle \\
&= \langle \delta(x \cdot y), \xi \otimes \eta \rangle + \langle (\mathrm{ad}(y) \otimes \mathrm{id}) \Delta(x), \xi \otimes \eta \rangle - \langle (\mathrm{id} \otimes L(x)) \delta(y), \xi \otimes \eta \rangle \\
&\quad + \langle (\mathrm{ad}(x) \otimes \mathrm{id}) \Delta(y), \xi \otimes \eta \rangle - \langle (\mathrm{id} \otimes L(y)) \delta(x), \xi \otimes \eta \rangle = 0,
\end{align*}
which is \eqref{eq:Bi4}. The other equivalences follow analogously by duality.
\end{proof}
 
\section{Dendrification of Almost Poisson Algebras}\label{Section4}
In this section we show that a weighted relative Rota-Baxter operator on an almost Poisson algebra gives rise to an almost tridendriform Poisson algebra; conversely, any almost tridendriform Poisson algebra naturally induces  almost Poisson algebra.
\begin{defn}[\cite{BaiGuoNi}]
Let $(A,\cdot)$ and $(V,\cdot_V)$ be two commutative associative algebras, and let $\mu:A\to\operatorname{End}(V)$ be a linear map such that $(V,\mu)$ is a representation of $(A,\cdot)$. Then $(V,\cdot_V,\mu)$ is called an \textit{$A$-module commutative associative algebra} if
\begin{equation}\label{eq:module-com-ass}
\mu(x)(a\cdot_V b)=(\mu(x)a)\cdot_V b,\qquad \forall x\in A,\;a,b\in V.
\end{equation}
\end{defn}

\begin{defn}
Let $(A,[\cdot,\cdot],\cdot)$ and $(V,[\cdot,\cdot]_V,\cdot_V)$ be two almost Poisson algebras, and let $\varrho,\mu:A\to\operatorname{End}(V)$ be linear maps satisfying:
\begin{enumerate}[label=\arabic*), leftmargin=*]
\item $(V,\cdot_V,\mu)$ is an $A$-module commutative associative algebra.
\item $(V,\mu,\varrho)$ is a representation of $A$.
\item For all $x,y\in A$ and $a,b\in V$, the following compatibilities hold:
\begin{align}
\varrho(x)(a\cdot_V b)&=(\varrho(x)a)\cdot_V b + a\cdot_V(\varrho(x)b),\label{eq:module-malpoi-1}\\
[a,\mu(x)b]_V &= -(\varrho(x)a)\cdot_V b + \mu(x)[a,b]_V.\label{eq:module-malpoi-3}
\end{align}
\end{enumerate}
Then $(V,\cdot_V,[\cdot,\cdot]_V,\mu,\varrho)$ is called an \textit{$A$-module almost Poisson algebra}.
\end{defn}

\begin{prop}
With the above notations, $(V,\cdot_V,[\cdot,\cdot]_V,\mu,\varrho)$ is an $A$-module almost Poisson algebra if and only if the direct sum $A\oplus V$ carries an almost Poisson algebra structure with products defined for all $x,y\in A$ and $a,b\in V$ by
\begin{align*}
\{x+a,\,y+b\} &= [x,y] + \varrho(x)b - \varrho(y)a + [a,b]_V,\\
(x+a)\bullet(y+b) &= x\cdot y + \mu(x)b + \mu(y)a + a\cdot_V b.
\end{align*}
This algebra is called the \textit{semi-direct product} and is denoted by $A\ltimes V$.
\end{prop}

\begin{proof}
It suffices to verify that the Jacobi identity \eqref{eq:malpoisalg} holds in $A\ltimes V$. Computing both sides explicitly:
\begin{align*}
&\{x+a,\,(y+b)\bullet(z+c)\}\\
&= \{x+a,\; y\cdot z + \mu(y)c + \mu(z)b + b\cdot_V c\}\\
&= [x,y\cdot z] + \varrho(x)\mu(y)c + \varrho(x)\mu(z)b + \varrho(x)(b\cdot_V c)\\
&\quad -\varrho(y\cdot z)a + [a,\mu(y)c]_V + [a,\mu(z)b]_V + [a,b\cdot_V c]_V,
\end{align*}
and
\begin{align*}
&\{x+a,y+b\}\bullet(z+c) + (y+b)\bullet\{x+a,z+c\}\\
&= \big([x,y]+\varrho(x)b-\varrho(y)a+[a,b]_V\big)\bullet(z+c)\\
&\quad + (y+b)\bullet\big([x,z]+\varrho(x)c-\varrho(z)a+[a,c]_V\big)\\
&= [x,y]\cdot z + \mu([x,y])c + \mu(z)\varrho(x)b - \mu(z)\varrho(y)a + \mu(z)[a,b]_V\\
&\quad + (\varrho(x)b)\cdot_V c - (\varrho(y)a)\cdot_V c + ([a,b]_V)\cdot_V c\\
&\quad + y\cdot[x,z] + \mu(y)\varrho(x)c - \mu(y)\varrho(z)a + \mu(y)[a,c]_V\\
&\quad + \mu([x,z])b + b\cdot_V(\varrho(x)c) - b\cdot_V(\varrho(z)a) + b\cdot_V([a,c]_V).
\end{align*}
The two expressions coincide precisely when $(V,\cdot_V,[\cdot,\cdot]_V,\mu,\varrho)$ satisfies the axioms of an $A$-module almost Poisson algebra. Hence $(A\ltimes V,\bullet,\{\cdot,\cdot\})$ is an almost Poisson algebra exactly under that condition.
\end{proof}

\begin{ex}
Let $(A,[\cdot,\cdot],\cdot)$ be an almost Poisson algebra. Then $(A,\operatorname{ad}_{[,]},L_\cdot)$ — where $\operatorname{ad}_{[,]}(x)y=[x,y]$ and $L_\cdot(x)y=x\cdot y$ — is an $A$-module almost Poisson algebra.
\end{ex}

\begin{defn}
A \textit{commutative dendriform trialgebra} is a $\mathbb{K}$-linear space $A$ equipped with two bilinear operations $\cdot$ and $\vartriangleright$ such that $(A,\cdot)$ is a commutative associative algebra and for all $x,y,z\in A$:
\begin{align}
(x\cdot y)\vartriangleright z &= x\vartriangleright(y\vartriangleright z),\label{eq:comm-dendr-1}\\
(x\vartriangleright y)\cdot z &= x\vartriangleright(y\cdot z).\label{eq:comm-dendr-2}
\end{align}
\end{defn}

\begin{prop}\label{prop:comm-dendr-to-ass}
Let $(A,\cdot,\vartriangleright)$ be a commutative dendriform trialgebra. Then the product
\begin{equation}\label{eq:sub-adjacent-ass}
x\circ y = x\vartriangleright y + y\vartriangleright x + x\cdot y
\end{equation}
defines a commutative associative algebra structure on $A$, denoted by $A^C$ and called the \textit{sub-adjacent commutative associative algebra}. Moreover, the map $\mathcal{L}_{\vartriangleright}:A\to\operatorname{End}(A)$ given by $\mathcal{L}_{\vartriangleright}(x)y = x\vartriangleright y$ endows $A$ with an $A^C$-module commutative associative algebra structure.
\end{prop}

\begin{defn}
An \textit{almost tridendriform  Poisson algebra} is a tuple $(A,[\cdot,\cdot],\diamond,\cdot,\vartriangleright)$ where:
\begin{itemize}
\item $(A,[\cdot,\cdot],\cdot)$ is an almost Poisson algebra;
\item $(A,\cdot,\vartriangleright)$ is a commutative dendriform trialgebra,
\end{itemize}
and the following compatibility conditions hold for all $x,y,z\in A$:
\begin{align}
x\diamond(y\cdot z) &= (x\diamond y)\cdot z + y\cdot(x\diamond z),\label{eq:post-cond-1}\\
[x,z\vartriangleright y] &= z\vartriangleright[x,y] - y\cdot(z\diamond x),\label{eq:post-cond-2}\\
(y\cdot z)\diamond x &= z\vartriangleright(y\diamond x) + y\vartriangleright(z\diamond x),\label{eq:post-cond-3}\\
\{x,z\}\vartriangleright y &= x\diamond(z\vartriangleright y) - z\vartriangleright(x\diamond y),\label{eq:post-cond-4}
\end{align}
where the operations $\cdot$ and $\{\cdot,\cdot\}$ are defined by
\begin{equation}\label{eq:induced-ops}
x\cdot y = x\vartriangleright y + y\vartriangleright x + x\cdot y,\qquad
\{x,y\} = x\diamond y - y\diamond x + [x,y].
\end{equation}
\end{defn}

\begin{re}
If in an almost tridendriform  Poisson algebra $(A,[\cdot,\cdot],\diamond,\cdot,\vartriangleright)$ the operations $[\cdot,\cdot]$ and $\cdot$ are trivial, then it reduces to an \textit{almost pre-Poisson algebra}, which generalizes the notion of pre-Poisson algebra introduced by M. Aguiar in \cite{Aguiar00}.
\end{re}

We now establish the fundamental relationship:

\begin{thm}
Let $(A,[\cdot,\cdot],\diamond,\cdot,\vartriangleright)$ be an almost tridendriform  Poisson algebra. Then $(A,\cdot,\{\cdot,\cdot\})$ — with $\cdot$ and $\{\cdot,\cdot\}$ given by \eqref{eq:induced-ops}   is an almost Poisson algebra, called the \textit{associated almost Poisson algebra} of $A$ and denoted by $A^C$.
\end{thm}

\begin{proof}
Proposition \ref{prop:comm-dendr-to-ass} guarantees that $(A,\cdot)$ is a commutative associative algebra. It remains to verify the Malcev identity \eqref{eq:malpoisalg} for $\{\cdot,\cdot\}$. A direct computation yields:
\begin{align*}
\{x,y\cdot z\} &= \{x,\, y\vartriangleright z + z\vartriangleright y + y\cdot z\}\\
&= x\diamond(y\vartriangleright z) + x\diamond(z\vartriangleright y) + x\diamond(y\cdot z)\\
&\quad - (y\vartriangleright z)\diamond x - (z\vartriangleright y)\diamond x - (y\cdot z)\diamond x\\
&\quad + [x,y\vartriangleright z] + [x,z\vartriangleright y] + [x,y\cdot z]\\
&= y\vartriangleright(x\diamond z) + \{x,y\}\vartriangleright z + z\vartriangleright(x\diamond y) + \{x,z\}\vartriangleright y\\
&\quad + (x\diamond y)\cdot z + y\cdot(x\diamond z) - z\vartriangleright(y\diamond x) - y\vartriangleright(z\diamond x)\\
&\quad + y\vartriangleright[x,z] - z\cdot(y\diamond x) + z\vartriangleright[x,y] - y\cdot(z\diamond x)\\
&\quad + [x,y]\cdot z + y\cdot[x,z]\\
&= y\vartriangleright\big(x\diamond z - z\diamond x + [x,z]\big) + \{x,z\}\vartriangleright y\\
&\quad + y\cdot\big(x\diamond z - z\diamond x + [x,z]\big) + \{x,y\}\vartriangleright z\\
&\quad + z\vartriangleright\big(x\diamond y - y\diamond x + [x,y]\big) + \big(x\diamond y - y\diamond x + [x,y]\big)\cdot z\\
&= y\vartriangleright\{x,z\} + \{x,z\}\vartriangleright y + y\cdot\{x,z\} + \{x,y\}\vartriangleright z + z\vartriangleright\{x,y\} + \{x,y\}\cdot z\\
&= y\cdot\{x,z\} + \{x,y\}\cdot z.
\end{align*}
Thus the Malcev identity holds, confirming that $(A,\cdot,\{\cdot,\cdot\})$ is an almost Poisson algebra.
\end{proof}

\begin{prop}
Let $(A,[\cdot,\cdot],\diamond,\cdot,\vartriangleright)$ be an almost tridendriform  Poisson algebra. Define $\mathcal{L}_\diamond,\mathcal{L}_\vartriangleright:A\to\operatorname{End}(A)$ by $\mathcal{L}_\diamond(x)y = x\diamond y$ and $\mathcal{L}_\vartriangleright(x)y = x\vartriangleright y$ for all $x,y\in A$. Then $(A,[\cdot,\cdot],\cdot,\mathcal{L}_\diamond,\mathcal{L}_\vartriangleright)$ is an $A$-module almost Poisson algebra of its associated almost Poisson algebra $(A^C,\cdot,\{\cdot,\cdot\})$.
\end{prop}

\begin{proof}
From Proposition \ref{prop:comm-dendr-to-ass}, $(A,\cdot,\mathcal{L}_\vartriangleright)$ is an $A^C$-module commutative associative algebra, and it is known that $(A,[\cdot,\cdot],\mathcal{L}_\diamond)$ is an $A^C$-module Malcev algebra. Axiom \eqref{eq:post-cond-1} translates directly into condition \eqref{eq:module-malpoi-1}, while \eqref{eq:post-cond-2} yields condition \eqref{eq:module-malpoi-3}. Hence the required structure is obtained.
\end{proof}

We now introduce the concept of a weighted relative Rota-Baxter operator for almost Poisson algebras, generalizing the weighted Rota–Baxter operators.

\begin{defn}
Let $(A,\cdot)$ be a commutative associative algebra and $(V,\cdot_V,\mu)$ an $A$-module commutative associative algebra. A linear map $\mathcal{R}:V\to A$ is called a \textit{weighted relative Rota-Baxter operator of weight $\lambda\in\mathbb{K}$} on $A$ with respect to $(V,\cdot_V,\mu)$ if
\begin{equation}\label{eq:kup-ass}
\mathcal{R}(a)\cdot\mathcal{R}(b) = \mathcal{R}\big(\mu(\mathcal{R}(a))b + \mu(\mathcal{R}(b))a + \lambda\,a\cdot_V b\big),\qquad \forall a,b\in V.
\end{equation}
\end{defn}

\begin{ex}
When $(V,\cdot_V,\mu) = (A,\cdot,L_\cdot)$, equation \eqref{eq:kup-ass} reduces to the weighted Rota–Baxter identity
\begin{equation}\label{eq:rba-ass}
\mathcal{R}(x)\cdot\mathcal{R}(y) = \mathcal{R}\big(\mathcal{R}(x)\cdot y + x\cdot\mathcal{R}(y) + \lambda x\cdot y\big),\quad x,y\in A.
\end{equation}
Thus $\mathcal{R}:A\to A$ is a weighted Rota–Baxter operator of weight $\lambda$ on the commutative associative algebra $(A,\cdot)$.
\end{ex}

\begin{lem}\label{lem:kup-to-dendr}
Let $(A,\cdot)$ be a commutative associative algebra and $(V,\cdot_V,\mu)$ an $A$-module commutative associative algebra. If $\mathcal{R}:V\to A$ is a weighted relative Rota-Baxter operator of weight $\lambda$ with respect to $(V,\cdot_V,\mu)$, then $V$ carries a commutative dendriform trialgebra structure defined for all $a,b\in V$ by
\begin{equation}\label{eq:dendr-from-kup}
a\cdot b = \lambda\,a\cdot_V b,\qquad a\vartriangleright b = \mu(\mathcal{R}(a))b.
\end{equation}
\end{lem}

\begin{defn}
Let $(A,[\cdot,\cdot],\cdot)$ be an almost Poisson algebra and $(V,[\cdot,\cdot]_V,\cdot_V,\varrho,\mu)$ an $A$-module almost Poisson algebra. A linear map $\mathcal{R}:V\to A$ is called a \textit{weighted relative Rota-Baxter operator of weight $\lambda\in\mathbb{K}$} with respect to $(V,[\cdot,\cdot]_V,\cdot_V,\varrho,\mu)$ if $\mathcal{R}$ satisfies \eqref{eq:kup-ass} (with the same $\mu$) together with the additional condition
\begin{equation}\label{eq:kup-lie}
[\mathcal{R}(a),\mathcal{R}(b)] = \mathcal{R}\big(\varrho(\mathcal{R}(a))b - \varrho(\mathcal{R}(b))a + \lambda\,[a,b]_V\big),\qquad \forall a,b\in V.
\end{equation}
\end{defn}

\begin{ex}
When $(V,[\cdot,\cdot]_V,\cdot_V,\varrho,\mu) = (A,[\cdot,\cdot],\cdot,\operatorname{ad}_{[,]},L_\cdot)$, the above definition yields a \textit{weighted Rota–Baxter operator of weight $\lambda$} on the almost Poisson algebra $(A,[\cdot,\cdot],\cdot)$.
\end{ex}

\begin{thm}\label{thm:kup-to-post}
Let $(A,\cdot,[\cdot,\cdot])$ be an almost Poisson algebra and $(V,[\cdot,\cdot]_V,\cdot_V,\varrho,\mu)$ an $A$-module almost Poisson algebra. Suppose $\mathcal{R}:V\to A$ is a weighted relative Rota-Baxter operator of weight $\lambda$ with respect to $(V,[\cdot,\cdot]_V,\cdot_V,\varrho,\mu)$. Define new operations on $V$ for all $a,b\in V$ by
\begin{align}
\{a,b\} &= \lambda[a,b]_V, &
a\diamond b &= \varrho(\mathcal{R}(a))b,\nonumber\\
a\cdot b &= \lambda a\cdot_V b, &
a\vartriangleright b &= \mu(\mathcal{R}(a))b.\label{eq:ops-from-kup}
\end{align}
Then $(V,\{\cdot,\cdot\},\diamond,\cdot,\vartriangleright)$ is an almost tridendriform  Poisson algebra. Moreover, $\mathcal{R}$ becomes a homomorphism of almost Poisson algebras from the associated almost Poisson algebra $V^c$ (of this almost tridendriform  Poisson algebra) to the original $(A,\cdot,[\cdot,\cdot])$.
\end{thm}

\begin{proof}
Lemma \ref{lem:kup-to-dendr} shows that $(V,\cdot,\vartriangleright)$ is a commutative dendriform trialgebra. We now verify the four compatibility conditions using the module axioms.

For \eqref{eq:post-cond-1}:
\begin{align*}
a\diamond(b\cdot c) &= \lambda\,\varrho(\mathcal{R}(a))(b\cdot_V c) \\
&= \lambda\big((\varrho(\mathcal{R}(a))b)\cdot_V c + b\cdot_V(\varrho(\mathcal{R}(a))c)\big) \quad\text{(by \eqref{eq:module-malpoi-1})}\\
&= (a\diamond b)\cdot c + b\cdot(a\diamond c).
\end{align*}

For \eqref{eq:post-cond-2}:
\begin{align*}
\{a,c\vartriangleright b\} &= \lambda[a,\mu(\mathcal{R}(c))b]_V \\
&= -\lambda(\varrho(\mathcal{R}(c))a)\cdot_V b + \lambda\mu(\mathcal{R}(c))[a,b]_V \quad\text{(by \eqref{eq:module-malpoi-3})}\\
&= -(c\diamond a)\cdot b + c\vartriangleright\{a,b\}.
\end{align*}

For \eqref{eq:post-cond-3}, using \eqref{eq:kup-ass} and the module property:
\begin{align*}
(b\cdot c)\diamond a &= (b\vartriangleright c + c\vartriangleright b + b\cdot c)\diamond a \\
&= \varrho\big(\mathcal{R}(\mu(\mathcal{R}(b))c + \mu(\mathcal{R}(c))b + \lambda b\cdot_V c)\big)a \\
&= \varrho(\mathcal{R}(b)\cdot\mathcal{R}(c))a \\
&= \mu(\mathcal{R}(c))\varrho(\mathcal{R}(b))a + \mu(\mathcal{R}(b))\varrho(\mathcal{R}(c))a \\
&= c\vartriangleright(b\diamond a) + b\vartriangleright(c\diamond a).
\end{align*}

For \eqref{eq:post-cond-4}, using \eqref{eq:kup-lie}:
\begin{align*}
\{a,c\}\vartriangleright b &= (a\diamond c - c\diamond a + [a,c])\vartriangleright b \\
&= \mu\big(\mathcal{R}(\varrho(\mathcal{R}(a))c - \varrho(\mathcal{R}(c))a + \lambda[a,c]_V)\big)b \\
&= \mu([\mathcal{R}(a),\mathcal{R}(c)])b \\
&= \varrho(\mathcal{R}(a))\mu(\mathcal{R}(c))b - \mu(\mathcal{R}(c))\varrho(\mathcal{R}(a))b \\
&= a\diamond(c\vartriangleright b) - c\vartriangleright(a\diamond b).
\end{align*}

Thus all conditions are satisfied, making $(V,\{\cdot,\cdot\},\diamond,\cdot,\vartriangleright)$ an almost tridendriform  Poisson algebra. Finally, the definitions \eqref{eq:ops-from-kup} together with \eqref{eq:kup-ass} and \eqref{eq:kup-lie} imply directly that $\mathcal{R}$ preserves both products, hence is a homomorphism of almost Poisson algebras from $V^c$ to $A$.
\end{proof}

\begin{ex}
Let $(A,\cdot,[\cdot,\cdot])$ be an almost Poisson algebra and $\mathcal{R}:A\to A$ a weighted Rota–Baxter operator of weight $\lambda$. Define new operations on $A$ for all $x,y\in A$ by
\begin{equation}\label{eq:rba-to-post}
\{x,y\} = \lambda[x,y],\quad x\diamond y = [\mathcal{R}(x),y],\quad x\cdot y = \lambda x\cdot y,\quad x\vartriangleright y = \mathcal{R}(x)\cdot y.
\end{equation}
Then $(A,\{\cdot,\cdot\},\diamond,\cdot,\vartriangleright)$ is an almost tridendriform  Poisson algebra, and $\mathcal{R}$ is a homomorphism from its associated almost Poisson algebra $A^c$ to the original $(A,\cdot,[\cdot,\cdot])$.
\end{ex}
\section{Embedding of Almost Poisson Algebras into \textsf{AWB}}\label{Section5}
 
In this section, we recall the notion of algebras with bracket (\textsf{AWB}), originally introduced and studied in \cite{casas}. We review their representations and characterize them via semi-direct products. We then show that an algebra with bracket can be viewed as a noncommutative generalization of an almost Poisson algebra through duplication by relative averaging operators.
 
\begin{defn}
A \emph{representation} of an associative algebra $(A, \cdot)$ is a vector space $V$ equipped with two linear maps $l, r: A \to End(V)$ such that
\begin{align*}
l(x \cdot y) &= l(x) l(y), \quad r(x \cdot y) = r(y) r(x), \quad l(x) r(y) = r(y) l(x)
\end{align*}
for all $x, y \in A$. Here, $(V, l, r)$ is also called an associative $A$-bimodule. If moreover $l = r$, we obtain a commutative associative $A$-module, which we denote by $(V, \mu)$.
\end{defn}
 
Note that the tuple $(V, l, r)$ is a representation of the associative algebra $(A, \cdot)$ if and only if $A \oplus V$ carries an associative structure given by
\begin{align}\label{sdpassociative}
(x + u) \cdot_{A \oplus V} (y + v) &= x \cdot y + l(x) v + r(y) u,
\end{align}
for all $x, y \in A$ and $u, v \in V$. This structure is called the \emph{semi-direct product} of $A$ with $V$.
 
In \cite{casas}, the authors introduced the notion of a (right) \textsf{AWB}. In this paper, we focus on the left case.
 
\begin{defn}
An \textsf{AWB} is a triple $(A, \cdot, \{\cdot, \cdot\})$, where $(A, \cdot)$ is an associative algebra equipped with a bracket $\{\cdot, \cdot\}$ such that the following identity holds:
\begin{align}\label{comleibpoisson}
\{x, y \cdot z\} = \{x, y\} \cdot z + y \cdot \{x, z\}, \quad \forall x, y, z \in A.
\end{align}
\end{defn}
 
If $(A, \cdot, \{\cdot, \cdot\})$ is a left \textsf{AWB}, then $(A, \cdot, \{\cdot, \cdot\}^{\mathrm{op}})$ is a right \textsf{AWB} (see \cite{casas}), where $\{x, y\}^{\mathrm{op}} = \{y, x\}$ for all $x, y \in A$.
 \begin{ex}
Let $A$ be a two-dimensional vector space with basis $\{e_1,e_2\}$. Define a bilinear product $\cdot$ on $A$ by
\begin{align*}
e_1 \cdot e_1 &= \alpha e_1, \\
e_1 \cdot e_2 &= \alpha e_2, \\
e_2 \cdot e_1 &= \beta e_1, \\
e_2 \cdot e_2 &= \beta e_2,
\end{align*}
and define a bracket $\{\cdot,\cdot\}$ on $A$ by
\begin{align*}
\{e_1,e_1\} &= \gamma e_1 - \frac{\alpha \gamma}{\beta} e_2,\\
\{e_1,e_2\} &= \nu e_1 - \frac{\alpha \nu}{\beta} e_2,\\
\{e_2,e_1\} &= \frac{\beta \gamma}{\alpha} e_1 - \gamma e_2,\\
\{e_2,e_2\} &= \frac{\beta \nu}{\alpha} e_1 - \nu e_2,
\end{align*}
where $\alpha\neq 0, \beta,\gamma$ and $\nu$ are parameters.  Then $(A,\cdot,\{\cdot,\cdot\})$ is an {\sf AWB}.
\end{ex}
\begin{re}
In an \textsf{AWB}, if the bracket is skew-symmetric and the associative structure is commutative, we recover the notion of an almost Poisson algebra. An almost Poisson algebra is a generalization of a Poisson algebra. We denote the categories of Poisson algebras, almost Poisson algebras, \textsf{AWB}s, and right algebras with bracket by \textsf{\bf PA}, \textsf{\bf APA}, \textsf{\bf AWB}, and \textsf{\bf rAWB}, respectively. Then we have the chain of embeddings
\[
\text{\textsf{\bf PA}} \hookrightarrow \text{\textsf{\bf APA}} \hookrightarrow \text{\textsf{\bf AWB}} \cong \text{\textsf{\bf rAWB}}.
\]
\end{re}
 
\begin{ex}
Let $A$ be a $2$-dimensional vector space with basis $\mathcal{B} = \{e_1, e_2\}$. Then $(A, \cdot, \{\cdot, \cdot\})$ is an \textsf{AWB}, where the multiplication $\cdot$ and bracket $\{\cdot, \cdot\}$ are defined by
\begin{align*}
e_1 \cdot e_1 &= e_1, \quad e_1 \cdot e_2 = e_2, \quad e_2 \cdot e_1 = 0, \quad e_2 \cdot e_2 = 0, \\
\{e_1, e_1\} &= 0, \quad \{e_1, e_2\} = e_2, \quad \{e_2, e_1\} = -e_2, \quad \{e_2, e_2\} = 0.
\end{align*}
\end{ex}
 
\begin{defn}
A subalgebra $S \subseteq A$ is called an \emph{algebra with bracket subalgebra} if:
\begin{enumerate}
\item $S$ is closed under the associative multiplication (i.e., if $x, y \in S$, then $x \cdot y \in S$);
\item $S$ is closed under the bracket (i.e., if $x, y \in S$, then $\{x, y\} \in S$).
\end{enumerate}
\end{defn}
 
\begin{defn}
Let $(A, \cdot, \{\cdot, \cdot\})$ be an \textsf{AWB}. The tuple $(V, l, r, L, R)$ is said to be a \emph{representation} of $A$ (or an $A$-module) if $(V, l, r)$ is a representation of the associative algebra $(A, \cdot)$ and $L, R: A \to End(V)$ are two linear maps satisfying
\begin{align}
L(x) l(y) &= l(\{x, y\}) + l(y) L(x), \label{cond rep1} \\
L(x) r(y) &= r(y) L(x) + r(\{x, y\}), \label{cond rep2} \\
R(x \cdot y) &= r(y) R(x) + l(x) R(y), \quad \forall x, y \in A. \label{cond rep3}
\end{align}
\end{defn}
 
\begin{prop}
Let $(A, \cdot, \{\cdot, \cdot\})$ be an \textsf{AWB}, $V$ a vector space, and $l, r, L, R: A \to  End(V)$ linear maps. Then $(V, l, r, L, R)$ is a representation of the \textsf{AWB} $(A, \cdot, \{\cdot, \cdot\})$ if and only if $(A \oplus V, \cdot_{A \oplus V}, \{\cdot, \cdot\}_{A \oplus V})$ is an \textsf{AWB}, where $\cdot_{A \oplus V}$ is given in \eqref{sdpassociative} and the bracket $\{\cdot, \cdot\}_{A \oplus V}$ is defined by
\begin{align}\label{sdpLeibniz}
\{x + u, y + v\}_{A \oplus V} &= \{x, y\} + L(x) v + R(y) u.
\end{align}
\end{prop}
 
\begin{proof}
According to \eqref{sdpassociative} (respectively, \eqref{sdpLeibniz}), $(A \oplus V, \cdot_{A \oplus V})$ is an associative algebra (respectively, $(A \oplus V, \{\cdot, \cdot\}_{A \oplus V})$ satisfies the bracket properties). It remains to verify the defining identity \eqref{comleibpoisson}. For any $x, y, z \in A$ and $u, v, w \in V$, we compute
\begin{align*}
&\{x + u, (y + v) \cdot_{A \oplus V} (z + w)\}_{A \oplus V} - \{x + u, y + v\}_{A \oplus V} \cdot_{A \oplus V} (z + w) - (y + v) \cdot_{A \oplus V} \{x + u, z + w\}_{A \oplus V} \\
&= \{x + u, y \cdot z + l(y) w + r(z) v\}_{A \oplus V} - (\{x, y\} + L(x) v + R(y) u) \cdot_{A \oplus V} (z + w) \\
&\quad - (y + v) \cdot_{A \oplus V} (\{x, z\} + L(x) w + R(z) u) \\
&= \{x, y \cdot z\} + L(x) l(y) w + L(x) r(z) v + R(y \cdot z) u - \{x, y\} \cdot z - l(\{x, y\}) w - r(z) L(x) v - r(z) R(y) u \\
&\quad - y \cdot \{x, z\} - l(y) L(x) w - l(y) R(z) u - r(\{x, z\}) v.
\end{align*}
This simplifies to
\begin{align*}
&L(x) l(y) w + L(x) r(z) v + R(y \cdot z) u - l(\{x, y\}) w - r(z) L(x) v\\&\quad\quad\quad\quad - r(z) R(y) u - l(y) L(x) w - l(y) R(z) u - r(\{x, z\}) v.
\end{align*}
Thus, $(A \oplus V, \cdot_{A \oplus V}, \{\cdot, \cdot\}_{A \oplus V})$ is an \textsf{AWB} if and only if $(V, l, r, L, R)$ is a representation of $(A, \cdot, \{\cdot, \cdot\})$.
\end{proof}
 
We now study relative averaging operators on almost Poisson algebras with respect to a given representation. These operators allow us to embed almost Poisson algebras into \textsf{AWB}s, and we provide some illustrative examples. First, we recall the notion of relative averaging operators (also known as embedding tensors) on commutative associative algebras (see \cite{aguiar}).
 
\begin{defn}
A linear map $\mathcal{K}: V \to A$ is called a \emph{relative averaging operator} on $(A, \c )$ with respect to a given representation $(V, \mu)$ if, for all $u, v \in V$, it satisfies
\begin{equation}\label{embass}
\mathcal{K}(u) \c  \mathcal{K}(v) = \mathcal{K}(\mu(\mathcal{K}(u)) v).
\end{equation}
\end{defn}
 
\begin{defn}
A relative averaging operator on a commutative associative algebra $(A, \c )$ with respect to the adjoint representation is called an \emph{averaging operator}. In this case, the identity \eqref{embass} can be rewritten as
\[
\mathcal{K}(x) \c  \mathcal{K}(y) = \mathcal{K}(\mathcal{K}(x) \c  y) = \mathcal{K}(x \c  \mathcal{K}(y)).
\]
\end{defn}
 
\begin{prop}\label{SDPandgraphassalg}
Let $(A, \c )$ be a commutative associative algebra and $(V, \mu)$ a representation. Then $(A \oplus V, \cdot_{A \oplus V})$ is an associative algebra, where
\begin{equation}\label{eq1}
(x + u) \cdot_{A \oplus V} (y + v) = x \c  y + \mu(x) v, \quad \forall x, y \in A, \, u, v \in V.
\end{equation}
This is called the \emph{hemisemi-direct product} associative algebra and is denoted by $A \oplus_\mu V$. Moreover, a linear map $\mathcal{K}: V \to A$ is a relative averaging operator on $(A, \c )$ if and only if the graph $\mathrm{Gr}(\mathcal{K}) = \{ \mathcal{K}(u) + u \mid u \in V \}$ is a subalgebra of the hemisemi-direct algebra $A \oplus_\mu V$.
\end{prop}
 
Let $(A, \cdot)$ be an associative algebra. Recall that a linear map $N: A \to A$ is said to be a \emph{Nijenhuis operator} on $A$ if it satisfies
\begin{align}\label{nij-ass}
N(x) \cdot N(y) = N(N(x) \cdot y + x \cdot N(y) - N(x \cdot y)).
\end{align}
 
\begin{prop}
Let $(A, \c )$ be a commutative associative algebra. A linear map $\mathcal{K}: V \to A$ is a relative averaging operator on $(A, \c )$ with respect to a representation $(V, \mu)$ if and only if the map
\[
N_{\mathcal{K}}: (A \oplus V) \to (A \oplus V), \quad x + u \mapsto \mathcal{K}(u)
\]
is a Nijenhuis operator on the hemisemi-direct product $A \oplus_\mu V$.
\end{prop}
 
\begin{prop}\label{commtoass}
Let $\mathcal{K}: V \to A$ be a relative averaging operator on a commutative associative algebra $(A, \c )$ with respect to a representation $(V, \mu)$. Define
\begin{align*}
u \cdot_{\mathcal{K}} v = \mu(\mathcal{K}(u)) v, \quad \forall u, v \in V.
\end{align*}
Then $(V, \cdot_{\mathcal{K}})$ is an associative algebra.
\end{prop}
 
\begin{proof}
For any $u, v, w \in V$, we have
\begin{align*}
(u \cdot_{\mathcal{K}} v) \cdot_{\mathcal{K}} w - u \cdot_{\mathcal{K}} (v \cdot_{\mathcal{K}} w) &= \mu(\mathcal{K}(\mu(\mathcal{K}(u)) v)) w - \mu(\mathcal{K}(u)) (\mu(\mathcal{K}(v)) w) \\
&\overset{\eqref{embass}}{=} \mu(\mathcal{K}(u) \c  \mathcal{K}(v)) w - \mu(\mathcal{K}(u)) (\mu(\mathcal{K}(v)) w) = 0.
\end{align*}
Thus, $(V, \cdot_{\mathcal{K}})$ is an associative algebra.
\end{proof}
 
\begin{defn}
Let $(A, \c , [\cdot, \cdot])$ be an almost Poisson algebra and $(V, \mu, \rho)$ a representation. A linear map $\mathcal{K}: V \to A$ is said to be a \emph{relative averaging operator} on the almost Poisson algebra $(A, \c , [\cdot, \cdot])$ with respect to $(V, \mu, \rho)$ if $\mathcal{K}$ is both a relative averaging operator on the commutative associative algebra $(A, \c )$ with respect to $(V, \mu)$ and satisfying 
$[\mathcal{K}(u) ,  \mathcal{K}(v)] = \mathcal{K}(\mu(\mathcal{K}(u)) v),\quad\forall u,v\in V.$
\end{defn}
 
\begin{defn}
A relative averaging operator on an almost Poisson algebra with respect to the adjoint representation is called an \emph{averaging operator}. The tuple $(A, \c , [\cdot, \cdot],\mathcal{K})$ is called averaging almost Poisson algebra.
\end{defn}
 
In the following, we provide characterizations of relative averaging operators on almost Poisson algebras using graphs and Nijenhuis operators.
 
\begin{prop}
Let $(A, \c , [\cdot, \cdot])$ be an almost Poisson algebra and $(V, \mu, \rho)$ a representation. Then $(A \oplus V, \cdot_{A \oplus V}, \{\cdot, \cdot\}_{A \oplus V})$ is an \textsf{AWB}, called the \emph{hemisemi-direct product} of the almost Poisson algebra $A$, and denoted by $A \oplus_{\mu, \rho} V$, where $\cdot_{A \oplus V}$ is given in \eqref{eq1} and $\{\cdot, \cdot\}_{A \oplus V}$ is defined by
\begin{equation}\label{eq2}
\{x + u, y + v\}_{A \oplus V} = [x, y] + \rho(x) v, \quad \forall x, y \in A, \, u, v \in V.
\end{equation}
\end{prop}
 
\begin{proof}
According to Proposition \ref{SDPandgraphassalg}, $(A \oplus V, \cdot_{A \oplus V})$ is an associative algebra. On the other hand, using the compatibility condition of the almost Poisson algebra and the representation properties \eqref{eq:repmalcpoisalg1}--\eqref{eq:repmalcpoisalg2}, for any $x, y, z \in A$ and $u, v, w \in V$, we have
\begin{align*}
&\{x + u, (y + v) \cdot_{A \oplus V} (z + w)\}_{A \oplus V} - \{x + u, y + v\}_{A \oplus V} \cdot_{A \oplus V} (z + w) \\
&\quad - (y + v) \cdot_{A \oplus V} \{x + u, z + w\}_{A \oplus V} \\
&= [x, y \c  z] + \rho(x) (\mu(y) w) - [x, y] \c  z - \mu([x, y]) w - y \c  [x, z] - \mu(y) (\rho(x) w) = 0.
\end{align*}
Thus, $(A \oplus V, \cdot_{A \oplus V}, \{\cdot, \cdot\}_{A \oplus V})$ is an \textsf{AWB}.
\end{proof}
 
\begin{thm}\label{graphMLie}
Let $(A, \c , [\cdot, \cdot])$ be an almost Poisson algebra and $(V, \mu, \rho)$ a representation. A linear map $\mathcal{K}: V \to A$ is a relative averaging operator if and only if the graph $\mathrm{Gr}(\mathcal{K}) = \{\mathcal{K}(u) + u \mid u \in V\}$ is an algebra with bracket subalgebra of the hemisemi-direct product $A \oplus_{\mu, \rho} V$.
\end{thm}
 
A \emph{Nijenhuis operator} on an \textsf{AWB} $(A, \cdot, \{\cdot, \cdot\})$ is a linear map $N: A \to A$ that is a Nijenhuis operator on the associative algebra $(A, \cdot)$ and satisfies
\begin{align}\label{nij-leibniz}
\{N(x), N(y)\} = N(\{N(x), y\} + \{x, N(y)\} - N(\{x, y\})).
\end{align}
 
\begin{prop}
Let $(A, \c , [\cdot, \cdot])$ be an almost Poisson algebra. A linear map $\mathcal{K}: V \to A$ is a relative averaging operator if and only if the map
\[
N_{\mathcal{K}}: (A \oplus V) \to (A \oplus V), \quad x + u \mapsto \mathcal{K}(u)
\]
is a Nijenhuis operator on the hemisemi-direct product $A \oplus_{\mu, \rho} V$.
\end{prop}
 
In the following theorem, we construct an \textsf{AWB} from a relative averaging operator on an almost Poisson algebra and interpret this \textsf{AWB} as a \emph{duplication} of the almost Poisson algebra.
 
\begin{thm}\label{poissontoNCLP}
Let $\mathcal{K}: V \to A$ be a relative averaging operator on an almost Poisson algebra $(A, \c , [\cdot, \cdot])$ with respect to a representation $(V, \mu, \rho)$. Then $V$ carries an \textsf{AWB} structure given by
\begin{align*}
u \cdot_{\mathcal{K}} v &= \mu(\mathcal{K}(u)) v, \qquad \{u, v\}_{\mathcal{K}} = \rho(\mathcal{K}(u)) v, \quad \forall u, v \in V.
\end{align*}
\end{thm}
 
\begin{proof}
According to Proposition \ref{commtoass}, $(V, \cdot_{\mathcal{K}})$ is an associative algebra. On the other hand, for any $u, v, w \in V$, we have
\begin{align*}
&\{u, v \cdot_{\mathcal{K}} w\}_{\mathcal{K}} - \{u, v\}_{\mathcal{K}} \cdot_{\mathcal{K}} w - v \cdot_{\mathcal{K}} \{u, w\}_{\mathcal{K}}\\ =& \{\ u, \mu(\mathcal{K}(v)) w \}_{\mathcal{K}} - \rho(\mathcal{K}(u)) v \cdot_{\mathcal{K}} w - v \cdot_{\mathcal{K}} \rho(\mathcal{K}(u)) w \\
=& \rho(\mathcal{K}(u)) (\mu(\mathcal{K}(v)) w) - \mu(\mathcal{K}(\rho(\mathcal{K}(u)) v)) w - \mu(\mathcal{K}(v)) (\rho(\mathcal{K}(u)) w) \\
=& \rho(\mathcal{K}(u)) (\mu(\mathcal{K}(v)) w) - \mu([\mathcal{K}(u), \mathcal{K}(v)]) w - \mu(\mathcal{K}(v)) (\rho(\mathcal{K}(u)) w) = 0.
\end{align*}
Hence, $(V, \cdot_{\mathcal{K}}, \{\cdot, \cdot\}_{\mathcal{K}})$ is an \textsf{AWB}.
\end{proof}
 
\begin{cor}
Let $\mathcal{K}: A \to A$ be an averaging operator on an almost Poisson algebra $(A, \c , [\cdot, \cdot])$. Then $(A, \cdot_{\mathcal{K}}, \{\cdot, \cdot\}_{\mathcal{K}})$ is an \textsf{AWB}, where
\begin{align*}
x \cdot_{\mathcal{K}} y &= \mathcal{K}(x) \c  y, \qquad \{x, y\}_{\mathcal{K}} = [\mathcal{K}(x), y], \quad \forall x, y \in A.
\end{align*}
\end{cor}
 
\begin{re}
As a consequence of the above results, we obtain a functor $F: \text{\bf aAPA} \to \text{\bf AWB}$ from the category \text{\bf aAPA} of averaging  almost Poisson algebras to the category \text{\bf AWB} of algebras with bracket.
\end{re}
 
\begin{ex}
Let $A$ be a $3$-dimensional vector space with basis $\mathcal{B} = \{e_1, e_2, e_3\}$. Then $(A, \c , [\cdot, \cdot])$ is an almost Poisson algebra, where the nonzero product $\c $ and bracket $[\cdot, \cdot]$ are defined by
\begin{align*}
e_1 \c  e_1 &= e_1, \quad e_1 \c  e_2 = e_2, \quad e_1 \c  e_3 = e_3, \\
[e_1, e_2] &= e_2, \quad [e_1, e_3] = -e_3, \quad [e_2, e_3] = e_1.
\end{align*}
Consider the averaging operator defined by $\mathcal{K}(e_1) = e_1$ (extended linearly). By the previous corollary, $(A, \cdot_{\mathcal{K}}, \{\cdot, \cdot\}_{\mathcal{K}})$ is an \textsf{AWB}, where
\begin{align*}
e_1 \cdot_{\mathcal{K}} e_1 &= e_1, \quad e_1 \cdot_{\mathcal{K}} e_2 = e_2, \quad e_1 \cdot_{\mathcal{K}} e_3 = e_3, \\
\{e_1, e_2\}_{\mathcal{K}} &= e_2, \quad \{e_1, e_3\}_{\mathcal{K}} = -e_3.
\end{align*}
\end{ex}


\begin{thebibliography}{99}

\bibitem{Abdaoui}
E. K. Abdaoui, S. Mabrouk and A. Makhlouf,
Rota–Baxter operators on pre-Lie superalgebras,
\emph{Bull. Malays. Math. Sci. Soc.} \textbf{42} (2019), no. 4, 1567–1606.

\bibitem{Aguiar00}
M. Aguiar,
Pre-Poisson algebras,
\emph{Lett. Math. Phys.} \textbf{54} (2000), 263–277.

\bibitem{Aguiar2}
M. Aguiar,
On the associative analog of Lie \(D\)-bialgebras,
\emph{J. Algebra} \textbf{244} (2001), 492–532.

\bibitem{Bai2010}
C. Bai,
Double construction of Frobenius algebras, Connes cocycles and their duality,
\emph{J. Noncommut. Geom.} (2010), 475–530.

\bibitem{BaiGuoNi}
C. Bai, L. Guo and X. Ni,
\(\mathcal{O}\)-operators on associative algebras and associative Yang–Baxter equations,
\emph{Pac. J. Math.} \textbf{256} (2012), 257–289.

\bibitem{MR3021790}
C. Bai, O. Bellier, L. Guo and X. Ni,
Splitting of operations, Manin products and Rota–Baxter operators,
\emph{Int. Math. Res. Not.} no. 3 (2013), 485–524.

\bibitem{Bax}
G. Baxter,
An analytic problem whose solution follows from a simple algebraic identity,
\emph{Pacific J. Math.} \textbf{10} (1960), 731–742.

\bibitem{MR0674005}
A. A. Belavin and V. G. Drinfeld,
Solutions of the classical Yang–Baxter equation for simple Lie algebras,
\emph{Funktsional. Anal. i Prilozhen.} \textbf{16} no. 3 (1982), 1–29, 96.

\bibitem{Bir1}
G. Birkhoff,
Moyennes des fonctions bornées,
Colloque d'algèbre et de théorie des nombres \textbf{24} (1949), 143–153.

\bibitem{Bir2}
G. Birkhoff,
Lattice theory,
\emph{Proc. Sympos. Pure Math.} \textbf{2} (1961), 163–172.

\bibitem{BCEM}
S. Braiek, T. Chtioui, M. Elhamdadi and S. Mabrouk,
Hom-associative algebras, admissibility and relative averaging operators,
\emph{arXiv:2304.12593}.

\bibitem{Safa}
S. Braiek, T. Chtioui and S. Mabrouk,
Anti-Leibniz algebras: a noncommutative version of mock-Lie algebras,
\emph{J. Geom. Phys.} \textbf{209} (2025), 105385.

\bibitem{SafaFatma}
S. Braiek, T. Chtioui, S. Mabrouk and F. Zouaidi,
Noncommutative Leibniz–Poisson algebras as simultaneous duplications of Poisson algebras and beyond,
(2026), submitted.

\bibitem{Casas1}
J. M. Casas,
Homology with trivial coefficients and universal central extension of algebras with bracket,
\emph{Comm. Algebra} \textbf{35} no. 8 (2007), 2431–2449.

\bibitem{Casas2}
J. M. Casas and T. Datuashvili,
Noncommutative Leibniz–Poisson algebras,
\emph{Comm. Algebra} \textbf{34} (2006), 2507–2530.

\bibitem{Casas3}
J. M. Casas and S. H. Jafari,
Isoclinism of algebras with bracket,
\emph{Bull. Iranian Math. Soc.} \textbf{51} no. 6 (2025), article no. 87.

\bibitem{Casas4}
J. M. Casas, E. Khmaladze and M. Ladra,
Wells-type exact sequence and crossed extensions of algebras with bracket,
\emph{Forum Math.} \textbf{36} no. 6 (2024), 1565–1584.

\bibitem{casas}
J. M. Casas and T. Pirashvili,
Algebras with bracket,
\emph{Manuscripta Math.} \textbf{119} no. 1 (2006), 1–15.

\bibitem{Chtioui1}
T. Chtioui, A. Hajjaji, S. Mabrouk and A. Makhlouf,
Cohomology and deformations of twisted \(\mathcal{O}\)-operators on 3-Lie algebras,
\emph{Filomat} \textbf{37} no. 21 (2023), 6977–6994.

\bibitem{Chtioui2}
T. Chtioui, S. Mabrouk and A. Makhlouf,
Cohomology and deformations of \(\mathcal{O}\)-operators on Hom-associative algebras,
\emph{J. Algebra} \textbf{604} (2022), 727–759.

\bibitem{CK}
A. Connes and D. Kreimer,
Renormalization in quantum field theory and the Riemann–Hilbert problem I,
\emph{Comm. Math. Phys.} \textbf{210} (2000), 249–273.

\bibitem{CannasWeinstein}
A. C. da Silva and A. Weinstein,
\emph{Geometric Models for Noncommutative Algebras},
Berkeley Math. Lecture Notes, Vol. 10, AMS, Providence, RI, 1999.

\bibitem{drinfeld}
V. G. Drinfeld,
Hamiltonian structures on Lie groups, Lie \(D\)-bialgebras and the classical Yang–Baxter equation,
\emph{Soviet Math. Dokl.} \textbf{27} (1983), 68–71.

\bibitem{FM}
R. A. S. Fox and J. B. Miller,
Averaging and Reynolds operators in Banach algebras III,
\emph{J. Math. Anal. Appl.} \textbf{24} (1968), 225–238.

\bibitem{G}
L. Guo,
\emph{An Introduction to Rota–Baxter Algebra},
Surveys of Modern Mathematics, Vol. 4, International Press, 2012.

\bibitem{HP}
C. B. Huijsmans and B. de Pagter,
Averaging operators and positive contractive projections,
\emph{J. Math. Anal. Appl.} \textbf{113} (1986), 163–184.

\bibitem{Kanatchikov}
I. V. Kanatchikov,
On field-theoretic generalizations of a Poisson algebra,
\emph{Rep. Math. Phys.} \textbf{40} no. 2 (1997), 225–234.

\bibitem{KKL}
I. Karimjanov, I. Kaygorodov and M. Ladra,
Rota-type operators on null-filiform associative algebras,
\emph{Linear Multilinear Algebra} \textbf{68} no. 3 (2018), 1–15.

\bibitem{K}
J. L. Kelley,
Averaging operators on \(C_0(X)\),
\emph{Illinois J. Math.} \textbf{2} (1958), 214–223.

\bibitem{KolesnikovMakarLimanov}
P. S. Kolesnikov, L. G. Makar-Limanov and I. P. Shestakov,
The Freiheitssatz for generic Poisson algebras,
\emph{SIGMA Symmetry Integrability Geom. Methods Appl.} \textbf{10} (2014), Paper 115.

\bibitem{Kupershmidt}
B. A. Kupershmidt,
What a classical \(r\)-matrix really is,
\emph{J. Nonlinear Math. Phys.} \textbf{6} (1999), 448–488.

\bibitem{Kuzmin}
E. N. Kuzmin,
Malcev algebras and their representations,
\emph{Algebra and Logic} \textbf{7} (1968), 233–244.

\bibitem{Sami}
S. Mabrouk,
Deformation of relative Rota-Baxter operators and Kupershmidt–Nijenhuis structures of a Malcev algebra,
\emph{Hacet. J. Math. Stat.} \textbf{51} (2021), 199–217.

\bibitem{Fatma}
S. Mabrouk, S. Silvestrov and F. Zouaidi,
Twisting \(\mathcal{O}\)-operators by \((2,3)\)-cocycles of Hom–Lie–Yamaguti algebras with representations,
\emph{J. Geom. Phys.} (2025), 105546.

\bibitem{PB}
J. Pei, C. Bai, L. Guo and X. Ni,
Replicators, Manin white product of binary operads and average operators,
in \emph{New Trends in Algebras and Combinatorics}, 2020, 317–353.

\bibitem{R}
O. Reynolds,
On the dynamical theory of incompressible viscous fluids,
\emph{Philos. Trans. Roy. Soc.} \textbf{186} (1895), 123–164.

\bibitem{R1}
G.-C. Rota,
On the representation of averaging operators,
\emph{Rend. Math. Univ. Padova} \textbf{30} (1960), 52–64.

\bibitem{R2}
G.-C. Rota,
Reynolds operators,
\emph{Proc. Sympos. Appl. Math.} \textbf{16} (1963), 70–83.

\bibitem{Rota}
G.-C. Rota,
Baxter algebras and combinatorial identities I, II,
\emph{Bull. Amer. Math. Soc.} \textbf{75} (1969), 325–329, 330–334.

\bibitem{MR2568415}
T. Schedler,
Poisson algebras and Yang–Baxter equations,
in \emph{Advances in quantum computation}, Contemp. Math. \textbf{482} (2009), 91–106.

\bibitem{S}
G. L. Seever,
Nonnegative projections on \(C_0(X)\),
\emph{Pacific J. Math.} \textbf{17} (1966), 159–166.

\bibitem{MR0725413}
M. A. Semenov-Tyan-Shanskii,
What a classical \(r\)-matrix is?,
\emph{Funktsional. Anal. i Prilozhen.} \textbf{17} no. 4 (1983), 17–33.

\bibitem{Shestakov}
I. P. Shestakov,
Speciality problem for Malcev algebras and Poisson Malcev algebras,
in \emph{Nonassociative Algebra and its Applications}, Lecture Notes Pure Appl. Math., Vol. 211, 2000, 365–371.

\end{thebibliography}
\end{document}